\documentclass[11pt]{amsart}
\usepackage[numbers,sort&compress]{natbib}
\usepackage[euler-digits]{eulervm}
\usepackage[a4paper,body={14.9cm,23cm},centering,includeheadfoot]{geometry} 
\usepackage{amssymb,float}
\usepackage{ctable}
\usepackage{longtable,aliascnt}
\usepackage[arc]{xy}
\usepackage[colorlinks=true, linkcolor=blue, breaklinks=true]{hyperref} 
\usepackage{breakurl}

\newtheorem{theorem}{Theorem}
\newaliascnt{proposition}{theorem}  
  
\aliascntresetthe{proposition}  
  
\newaliascnt{lemma}{theorem}  
\newtheorem{lemma}[lemma]{Lemma}
\aliascntresetthe{lemma}  

\newtheorem{corollary}[theorem]{Corollary}
\def\sectionautorefname~#1\null{\S#1\null}
\def\subsectionautorefname~#1\null{\S#1\null}

\def\equationautorefname~#1\null{(#1)\null}
\def\itemautorefname~#1\null{(#1)\null}

\newcommand{\GAP}{\textsf{GAP}} 
\newcommand{\CHEVIE}{\textsf{CHEVIE}}
\newcommand{\ZigZag}{\textsf{ZigZag}} 
\numberwithin{equation}{section}

\newcommand\mubar{\overline{\mu}}
\renewcommand\th{^{\text{th}}}

\newcommand\ibar{\overline{\imath}}
\newcommand\jbar{\overline{\jmath}}
\newcommand\BBC{\mathbb C}
\newcommand\BBN{\mathbb N}

\title[Computations for rank seven and eight Coxeter groups]
{Computations for Coxeter arrangements and Solomon's descent algebra
  III: Groups of rank seven and eight} 



\begin{document}
\author[M. Bishop]{Marcus Bishop} 
\address{Fakult\"at f\"ur
  Mathematik\\Ruhr-Universit\"at Bochum\\D-44780 Bochum, Germany}
\email{marcus.bishop@rub.de} 

\author[J. M. Douglass]{J. Matthew Douglass} 
\address{Department of Mathematics\\University of North Texas\\Denton
  TX, USA 76203} 
\email{douglass@unt.edu}

\author[G. Pfeiffer]{G\"otz Pfeiffer} 
\address{School of Mathematics, Statistics and Applied
  Mathematics\\National University of Ireland, Galway\\University
  Road, Galway, Ireland} 
\email{goetz.pfeiffer@nuigalway.ie}

\author[G. R\"ohrle]{Gerhard R\"ohrle} 
\address{Fakult\"at f\"ur Mathematik\\Ruhr-Universit\"at
  Bochum\\D-44780 Bochum, Germany} 
\email{gerhard.roehrle@rub.de}

\begin{abstract}
  In this paper we extend the computations in parts I and II of this series of 
  papers and complete the proof of a conjecture of Lehrer and Solomon expressing
  the character of a finite Coxeter group $W$ acting on the $p^{\text{th}}$
  graded component of its Orlik-Solomon algebra as a sum of characters
  induced from linear characters of centralizers of elements of $W$ for
  groups of rank seven and eight. For classical Coxeter groups, these
  characters are given using a formula that is expected to hold in all
  ranks.
\end{abstract}

\maketitle

\section{Introduction}

Suppose that $W$ is a finite Coxeter group and that $V$ is the complexified
reflection representation of $W$.  Let $\mathcal A$ be the set of reflecting
hyperplanes of $W$ in $V$ and let
\[
M=V\setminus \bigcup _{H\in\mathcal A} H
\]
denote the complement of these hyperplanes in $V$. The reflection length of
an element $w \in W$ is the least integer $p$ such that $w$ may be written
as a product of $p$ reflections. Clearly, conjugate elements have the same
reflection length. Lehrer and Solomon \cite[1.6]{lehrersolomon:symmetric}
conjectured that there is a $\mathbb CW$-module isomorphism
\begin{equation}
  \label{eq:lsconj}
  H^p(M, \mathbb C)\cong \bigoplus_c \operatorname{Ind}_{C_W(c)}^W
  \chi_c \qquad p=0, \dots, \mathrm{rank}(W)
\end{equation}
where $c$ runs over a set of representatives of the conjugacy classes of $W$
with reflection length equal $p$ and $\chi_c$ is a suitable linear
character of the centralizer $C_W(c)$ of $c$ in $W$. Lehrer and Solomon
proved~\autoref{eq:lsconj} for symmetric
groups. In~\cite{douglasspfeifferroehrle:inductive} we proposed an inductive
approach to the Lehrer-Solomon conjecture that establishes a direct
connection between the character of the Orlik-Solomon algebra of $W$ and the
regular character of $W$. This inductive approach has been used to prove
~\autoref{eq:lsconj} for dihedral
groups~\cite{douglasspfeifferroehrle:inductive}, symmetric groups
\cite{douglasspfeifferroehrle:cohomology}, and Coxeter groups with rank at
most six~\cite{bishopdouglasspfeifferroehrle:computations,
  bishopdouglasspfeifferroehrle:computationsII}.

In this paper we extend the computations
in~\cite{bishopdouglasspfeifferroehrle:computations,
  bishopdouglasspfeifferroehrle:computationsII} to finite Coxeter groups of
rank seven and eight and complete the proof of the conjectures
in~\cite{douglasspfeifferroehrle:inductive} that relate the Orlik-Solomon
and regular characters of these groups. As a consequence, the conjectures
in~\cite{douglasspfeifferroehrle:inductive}, as well as the Lehrer-Solomon
conjecture, are shown to hold for all finite Coxeter groups of rank at most
eight. In particular, these conjectures hold for all exceptional finite
Coxeter groups.

As described in more detail below, the proof for the Coxeter groups of type
$E_7$ and $E_8$ uses the techniques developed
in~\cite{bishopdouglasspfeifferroehrle:computations,
  bishopdouglasspfeifferroehrle:computationsII}, except that we use
Fleischmann and Janiszczak's
computation~\cite{fleischmannjaniszczak:combinatorics} of the M\"obius
functions of fixed point sets in the intersection lattice of $\mathcal A$ to
compute the character of the top component of the Orlik-Solomon algebra. We
take this opportunity to correct several minor misprints in the table for
$E_8$ in~\cite{fleischmannjaniszczak:combinatorics}. For
Coxeter groups of classical types in this paper we give explicit formulas
for the characters $\chi_c$ in all ranks and then use the methods developed
in~\cite{bishopdouglasspfeifferroehrle:computations,
  bishopdouglasspfeifferroehrle:computationsII} to verify
that~\autoref{eq:lsconj} holds for rank less than or equal eight. The
formulas for the characters $\chi_c$ given below are similar to the
formulas in \cite{lehrersolomon:symmetric}.

The rest of this paper is organized as follows. In \autoref{sec:state} we
establish notation, give a concise review of the constructions and
conjectures in~\cite{douglasspfeifferroehrle:inductive,
  bishopdouglasspfeifferroehrle:computations,
  bishopdouglasspfeifferroehrle:computationsII}, and state the main theorem
to be proved in this paper (\autoref{thm:1}); in \autoref{sec:class} we give
explicit constructions of the linear characters that we expect to satisfy
the conclusion of~\autoref{thm:1} for classical groups, and we verify that
these characters do indeed satisfy the conclusion of~\autoref{thm:1} for
groups of type $B_n$ and $D_n$ for $n\leq 8$; finally, in \autoref{sec:exc}
we give specific linear characters that satisfy the conclusion
of~\autoref{thm:1} for the exceptional groups of type $E_7$ and $E_8$, thus
completing the proof of the theorem.

\section{Statement of the main theorem}\label{sec:state}

We begin by summarizing the constructions
in~\cite{douglasspfeifferroehrle:inductive,
  bishopdouglasspfeifferroehrle:computations,
  bishopdouglasspfeifferroehrle:computationsII}. The reader is referred to
these sources for more details and proofs. Let $\left(W,S\right)$ be a
finite Coxeter system and let $A\left(W\right)$ be the Orlik-Solomon algebra
of $W$. The inductive strategy for proving~\autoref{eq:lsconj} proposed
in~\cite{douglasspfeifferroehrle:inductive} is to decompose the left regular
$\mathbb C W$-module and the Orlik-Solomon algebra of $W$ (considered as a
left $\mathbb C W$-module) into direct sums, and then relate the characters
of the individual summands. The decomposition of $\mathbb{C}W$ is given by a
set of orthogonal idempotents $\{\, e_\lambda\mid \lambda \in \Lambda\,\}$
constructed by Bergeron, Bergeron, Howlett and Taylor
\cite{bergeronbergeronhowletttaylor:decomposition}. Here $\Lambda$ denotes
the set of \emph{shapes} of $W$, that is, the set of subsets of $S$ modulo
the equivalence relation given by conjugacy in $W$.  Alternately, $\Lambda$
indexes the conjugacy classes of parabolic subgroups of $W$. For each subset
$L$ of $S$, Bergeron et al.~construct a quasi-idempotent $e_L$ in the descent
algebra of $W$ and define $e_\lambda=\sum_{L\in \lambda} e_L$. Denoting the
regular character of $W$ by $\rho$ and the character of $\mathbb{C}W
e_\lambda$ by $\rho_\lambda$, it follows that
\begin{equation}
  \label{eq:1a}
  \rho=\sum_{\lambda\in\Lambda}\rho_\lambda.
\end{equation}

The set of shapes $\Lambda$ also indexes the orbits of $W$ on the lattice of
$\mathcal A$ and a construction of Lehrer and Solomon
\cite[\S2]{douglasspfeifferroehrle:cohomology} yields a decomposition
$A\left(W\right) =\bigoplus_{ \lambda\in\Lambda} A_\lambda$.  Denoting the
Orlik-Solomon character of $W$ by $\omega$ and the character of $A_\lambda$
by $\omega_\lambda$, we have
\begin{equation}
  \label{eq:OmegaSum}
  \omega=\sum_{\lambda\in\Lambda}\omega_\lambda.
\end{equation}

Suppose that $L\subseteq S$.  Then $(W_L,L)$ is a Coxeter system and we may
consider the Orlik-Solomon algebra $A\left(W_L\right)$ of $W_L$. For
$J\subseteq L$ we denote by $e^L_J$ the quasi-idempotents in $\mathbb{C}W_L$
constructed in \cite{bergeronbergeronhowletttaylor:decomposition}. The
homogeneous component of $A\left( W_L \right)$ of highest degree is called
the {\em top component} of $A\left(W_L\right)$.  By analogy, the submodule
$\mathbb{C}W_L e_L^L$ of $\mathbb{C}W_L$ is called the {\em top component}
of $\mathbb{C}W_L$. The top components of $A\left(W_L\right)$ and
$\mathbb{C}W_L$ are $N_W(W_L)$-stable subspaces of $A\left(W\right)$ and
$\mathbb{C}W$, respectively. We denote their characters by
$\widetilde{\omega_L}$ and $\widetilde{\rho_L}$.  If $\lambda$ is in
$\Lambda$ and $L$ is in $\lambda$, then
\begin{equation}
  \label{eq:3a}
  \omega_\lambda=\operatorname{Ind}_{N_W\left(W_L\right)}^W
  \widetilde{\omega_L}\quad\text{and} \quad\rho_\lambda=
  \operatorname{Ind}_{N_W\left(W_L\right)}^W \widetilde{\rho_L}
\end{equation}
(see \cite[Proposition~4.8]{douglasspfeifferroehrle:cohomology}).  In order
to state our main result we need to recall two more definitions. First, an
element in $W_L$ is \emph{cuspidal} if it does not lie in any proper
parabolic subgroup of $W_L$. Second, for $n$ in $N_W(W_L)$, let
$\alpha_L(n)$ be the determinant of the restriction of $n$ to the space of
fixed points of $W_L$ in $V$. Note that $\alpha_L$ is a linear character of
$N_W\left(W_L\right)$.

\begin{theorem}\label{thm:1}
  Suppose that $\left(W,S\right)$ is a finite Coxeter system of rank at most
  eight. Let $L\subseteq S$ and suppose that $\mathcal{C}_L$ is a set of
  representatives of the cuspidal conjugacy classes of $W_L$. Then for each
  $w\in\mathcal{C}_L$ there exists a linear character $\varphi_w$ of
  $C_W\left(w\right)$ such that
  \[
  \widetilde{\rho_L}=\sum_{w\in\mathcal{C}_L}
  \operatorname{Ind}_{C_W(w)}^{N_W(W_L)} \varphi_w =\alpha_L\epsilon
  \widetilde{\omega_L},
  \]
  where $\epsilon$ is the sign character of $W$. 
\end{theorem}

Set $\alpha_w=\alpha_L|_{C_W(w)}$ when $w$ is cuspidal in $W_L$. Then it
follows from~\autoref{eq:3a} and the theorem that for $\lambda$ in $\Lambda$
we have
\begin{equation*}
  \rho_\lambda= \operatorname{Ind}_{N_W\left(W_L\right)}^W
  \widetilde{\rho_L} = \operatorname{Ind}_{N_W\left(W_L\right)}^W \Big(
  \sum_{w\in\mathcal{C}_L} \operatorname{Ind}_{C_W(w)}^{N_W(W_L)} \varphi_w
  \Big)\\ 
  = \sum_{w\in\mathcal{C}_L} \operatorname{Ind}_{C_W(w)}^{W} \ \varphi_w  
\end{equation*}
and 
\begin{multline}
  \label{eq:ro}
  \omega_\lambda= \operatorname{Ind}_{N_W\left(W_L\right)}^W
  \widetilde{\omega_L} = \operatorname{Ind}_{N_W\left(W_L\right)}^W \Big(
  \sum_{w\in\mathcal{C}_L} \operatorname{Ind}_{C_W(w)}^{N_W(W_L)} (\alpha_w
  \epsilon \varphi_w) \Big)\\
  = \sum_{w\in\mathcal{C}_L} \operatorname{Ind}_{C_W(w)}^{W} (\alpha_w
  \epsilon \varphi_w).
\end{multline}
Then the Lehrer-Solomon conjecture
for finite Coxeter groups with rank at most eight follows immediately from
\autoref{eq:ro} and~\autoref{eq:3a} because
\[
H^p(M, \mathbb C)\cong \bigoplus_{\mathrm{rank}(\lambda)=p} A_\lambda,
\]
where $\mathrm{rank}(\lambda)= |L|$ for any $L$ in $\lambda$ (see
\cite[\S2]{douglasspfeifferroehrle:cohomology}).

Using the notation above, set $\chi_w= \alpha_w\epsilon\varphi_w$. 

\begin{corollary}
  Suppose that $W$ is a finite Coxeter group of rank at most eight. Then
  there is a $\mathbb CW$-module isomorphism
  \begin{equation*}
    H^p(M, \mathbb C)\cong \bigoplus_w \operatorname{Ind}_{C_W(w)}^W
    (\chi_w) \qquad p=0, \dots, \mathrm{rank}(W)
  \end{equation*}
  where $w$ runs over a set of representatives of the conjugacy classes of
  $W$ with reflection length equal $p$ and $\chi_w$ is a linear character
  of $C_W(w)$.
\end{corollary}

Using ~\autoref{eq:1a},~\autoref{eq:OmegaSum}, and~\autoref{eq:3a}, the
theorem yields the following corollary, which relates the Orlik-Solomon and
the regular characters of $W$.

\begin{corollary}\label{ConjectureA}
  Suppose that $W$ is a finite Coxeter group of rank at most eight and that
  $\mathcal{R}$ is a set of conjugacy class representatives of $W$. Then for
  each $w\in\mathcal{R}$ there exists a linear character $\varphi_w$ of
  $C_W(w)$ such that
  \[
  \rho= \sum_{w\in\mathcal{R}} \operatorname{Ind}_{C_W(w)}^W \varphi_w
  \quad\text{and}\quad\omega= \epsilon \sum_{w\in\mathcal{R}}
  \operatorname{Ind}_{C_W(w)}^W (\alpha_w \varphi_w ),
  \] 
  where $\epsilon$ is the sign character of $W$.
\end{corollary}

As noted above,  \autoref{thm:1} has been proved for symmetric groups,
dihedral groups, and Coxeter groups of rank at most six in earlier work. It
is also shown in~\cite{douglasspfeifferroehrle:cohomology} that if the
conclusion of~\autoref{thm:1} holds for Coxeter groups $W$ and $W'$ then it
holds for $W\times W'$. Therefore, it suffices to consider only irreducible
Coxeter groups.  In this paper we complete the proof of the theorem by
showing that the conclusion of~\autoref{thm:1} holds for Coxeter groups of
type $B_7$, $B_8$, $D_7$, $D_8$, $E_7$, and $E_8$.

To prove~\autoref{thm:1} for the groups just listed, we follow the approach
described in \cite[\S4]{douglasspfeifferroehrle:inductive} using the \GAP\
computer algebra system system \cite{gap3} with the \CHEVIE\ \cite{chevie}
and \ZigZag\ \cite{zigzag} packages.
\begin{enumerate} \setlength{\itemindent}{2.5em} 
\item[\scshape Step 1.]\label{Step1} Compute $\widetilde{\rho_L}$,
  $\widetilde{\omega_L}$, and verify that $\widetilde{\rho_L}$ and
  $\alpha_L\epsilon \widetilde{\omega_L}$ are equal.
\item[\scshape Step 2.]\label{Step2} Find linear characters $\varphi_w$ such
  that
  \begin{equation*}
    \widetilde{\rho_L}=\sum_{w\in\mathcal{C}_L}
    \operatorname{Ind}_{C_W(w)}^{N_W(W_L)} \varphi_w.    
  \end{equation*}
\end{enumerate}

The character of $W_L$ afforded by $\BBC W_L e_L^L$ is denoted by
$\rho_L$. Then $\widetilde{\rho_L}$ is an extension of $\rho_L$ to
$N_W(W_L)$. Similarly, the character of $W_L$ on $A(W_L)$ is denoted by
$\omega_L$, and $\widetilde{\omega_L}$ is the extension of $\omega_L$ to
$N_W(W_L)$. The computation of the top component characters $\rho_S=
\widetilde{\rho_S}$ is described
in~\cite[\S3.1]{bishopdouglasspfeifferroehrle:computations}. This
computation has been implemented in the {\tt ECharacters} function in the
\ZigZag\ package. The {\tt ECharacters} function takes a finite Coxeter
group as its argument and returns the list of the characters
$\rho_\lambda$. Then $\rho_{\{S\}}= \widetilde{\rho_S}$. For $L$ a proper
subset of $S$, the computation of $\widetilde{\rho_L}$, given $\rho_L$, is
described in ~\cite[\S2]{bishopdouglasspfeifferroehrle:computationsII}.

The computation of the characters $\widetilde{\omega_L}$, for $L$ a proper
subset of $S$, is described
in~\cite[\S3]{bishopdouglasspfeifferroehrle:computations}
and~\cite[\S2]{bishopdouglasspfeifferroehrle:computationsII}. The method
used in these references is computationally too expensive to compute the
character $\omega_S$ for the groups of rank eight. In this paper we take an
alternate approach.

Orlik and Solomon~\cite{orliksolomon:combinatorics} proved that the graded
character of an element $w$ in $W$ on $A(W)$ may be computed using the
M\"obius function of the poset of fixed points of $w$ in the lattice of
$\mathcal A$. Precisely, let $\mathcal{L}$ be the intersection lattice of
$\mathcal{A}$ and for $w$ in $W$, let $\mathcal{L}^w$ denote the subposet of
$w$-stable subspaces in $\mathcal{L}$. Then
\begin{equation}
  \label{eq:poin}
  \sum_{i=0}^n
  \operatorname{Trace} \left(w, H^i\left( M,\mathbb{C} \right) \right) t^i
  =\sum_{ X\in\mathcal{L}^w} \mu_w \left( X\right) \left(-t \right)^{n- \dim{X}},  
\end{equation}
where $\mu_w$ is the M\"obius function of $\mathcal{L}^w$, $n$ is the rank
of $W$, and $t$ is an indeterminate. Denote the polynomial
in~\autoref{eq:poin} by $P_w(t)$. Then $\omega_S\left(w \right)$ is the
coefficient of $t^n$ in $P_w\left(t\right)$.

The polynomials $P_w(t)$ have been computed by Lehrer for Coxeter groups of
types $A$ and $B$~\cite{lehrer:poincare, lehrer:hyperoctahedral}, and
Fleischmann and Janiszczak in all
cases~\cite{fleischmannjaniszczak:lattices,
  fleischmannjaniszczak:combinatorics}. When $W$ has type $E_7$ or $E_8$ we
used the polynomials calculated by Fleischmann and Janiszczak to find
$\omega_S$. For the groups of type $B_7$, $D_7$, $B_8$, and $D_8$ we
calculated the polynomials $P_w\left(t\right)$ as described below.

The first step is to calculate $\mathcal{L}$.  The subspaces in
$\mathcal{L}$ are parameterized by the set of pairs $\left(\tau, L\right)$
where $L\subseteq S$ runs through a fixed set of representatives of the
shapes of $W$ and $\tau\in W$ is a coset representative of
$N_W\left(W_L\right)$ in $W$.  The pair $\left(\tau, L\right)$ corresponds
with the subspace $\tau X_L$ of $V$, where $X_L=\bigcap_{s\in L}
\operatorname{Fix} \left(s\right)$.

The subspace corresponding to $\left(\tau, L\right)$ is contained in the
subspace corresponding to $\left(\sigma, K\right)$ if and only if $\tau X_L
\subseteq\sigma X_K$.  This in turn holds if and only if $\sigma^{-1} \tau
X_L\subseteq X_K$ which holds if and only if $W_K\subseteq
W_L^{\tau^{-1}\sigma}$.  This last condition can be checked by calculating a
minimal length representative $z$ of the $\left(W_L, W_K\right)$-double
coset of $\tau^{-1}\sigma$. Then $W_K\subseteq W_L^{\tau^{-1} \sigma}$ if
and only if $K\subseteq L^z$.  We also remark that it suffices to assume
that $\sigma=1$, for the spaces contained in $\left(\sigma, K\right)$ are
precisely the spaces $\left(\sigma\tau, L\right)$ for which $\left(\tau,
  L\right)$ is contained in $\left(1, K\right)$.

Finally, to determine which subspaces $\tau X_L$ are in the subposet
$\mathcal{L}^w$, we observe that $w\tau X_L=\tau X_L$ if and only if
$\tau^{-1}w\tau X_L=X_L$. This in turn holds if and only if $\tau^{-1}w
\tau\in N_W\left(W_L\right)$.  It remains to calculate $P_w\left(t\right)$
using the formula above, after calculating all the values of $\mu$ by
recursion. Note that if $w_0$ is central in $W$ then
$\mathcal{L}^w=\mathcal{L}^ {w_0w}$. In this situation, only one of
$P_w\left(t \right)$ or $P_{w_0w}\left(t\right)$ needs to be calculated
whenever $w$ and $w_0w$ lie in different conjugacy classes.

We remark that the polynomials calculated using the method above in
types~$B_7$ and $B_8$ agree with those given by Lehrer
\cite{lehrer:hyperoctahedral}. We also note that in the table for $E_8$ in
\cite{fleischmannjaniszczak:combinatorics} the values for the classes called
$2A_2$ and $2D_4$ are missing a minus sign. Also, the second appearance of
$A_1+E_6(a_2)$ should read $A_2+E_6(a_2)$ and the correct polynomial for the
class $E_8(a_3)$ is $(t-1)(t+1)(t^2+1)(13t^4-1)$.

\hyperref[Step1]{\scshape Step~1} is completed by comparing
$\widetilde{\rho_L}$ and $\alpha_L\epsilon \widetilde{\omega_L}$ once both
have been computed.  It remains to
complete~\hyperref[Step2]{\scshape Step~2}: find linear characters
$\varphi_w$ such that
\[
\widetilde{\rho_L} =\sum_{w\in \mathcal{C}_L} \operatorname{Ind}_{C_W
  (w)}^{N_W (W_L)} \varphi_w.
\]
This is accomplished in the next section for the groups of type $B_n$ and
$D_n$, and in the final section for the groups of type $E_7$ and $E_8$.

\section{Classical groups}\label{sec:class}

In this section we take $W$ to be of classical type. For each subset $L$ of
$S$ and each cuspidal element $w$ in $W_L$ we construct a linear character
$\varphi_w$ of $C_W (w)$. We then verify that when the rank of $W$ is at
most eight, these characters satisfy the conclusion of~\autoref{thm:1}. For
symmetric groups, the characters $\varphi_w$ are the ones constructed
in~\cite{douglasspfeifferroehrle:cohomology}, where it is shown that they
satisfy the conclusion of~\autoref{thm:1}. We expect that the characters
$\varphi_w$ constructed in this section satisfy the conclusion
of~\autoref{thm:1} for all ranks.

Note that because $\widetilde{\rho_L}$ only depends on the conjugacy
class of $W_L$, or equivalently, the shape of $L$, we need only consider a
suitably chosen representative in each conjugacy class of parabolic
subgroups of $W$, and because $\operatorname{Ind}_{C_W (w)}^{N_W (W_L)}
\varphi_w$ depends only in the conjugacy class of $w$, we need only consider
one suitably chosen representative in each cuspidal conjugacy class of
$W_L$.

Our construction follows the same general pattern as the construction
in~\cite{douglasspfeifferroehrle:cohomology}, which goes back at least
to~\cite{lehrersolomon:symmetric}. It is most naturally phrased in terms of
permutations and signed permutations. We begin by reviewing the construction
for symmetric groups.

\subsection{The characters \texorpdfstring{$\varphi^A_{\lambda}$}{}}

A \emph{composition} is a non-empty tuple $\lambda= (\lambda_1, \dots,
\lambda_a)$ of positive integers and a \emph{partition} is a composition
with the property that $\lambda_i\geq \lambda_{i+1}$ for $1\leq i\leq
a-1$. The numbers $\lambda_i$ are called the \emph{parts} of $\lambda$, the
sum of the parts of $\lambda$ is denoted by $|\lambda|$, and the number of
parts of $\lambda$ is denoted by $l(\lambda)$. If $|\lambda|=n$, then
$\lambda$ is called a composition, or a partition, of $n$. If $\lambda$ is a
partition of $n$, we write $\lambda \vdash n$. By convention, the empty
tuple is a composition, and a partition, of zero.

For a positive integer $n$, set $[n]=\{1, 2, \dots, n\}$ and let $S_n$ denote
the group of permutations of $[n]$. If $s_i$ denotes the transposition that
switches $i$ and $i+1$, then $S=\{s_1, s_2, \dots, s_{n-1}\}$ is a Coxeter
generating set for $S_n$ such that $(S_n, S)$ is a Coxeter system of type
$A_{n-1}$. The product of symmetric groups $S_{\lambda_1} \times \dotsm
\times S_{\lambda_a}$ is considered as a subgroup, $S_\lambda$, of the
symmetric group $S_n$ via the obvious embedding, where $S_{\lambda_1}$ acts
on $\{1, \dots, \lambda_1\}$, $S_{\lambda_2}$ acts on $\{\lambda_1+1, \dots,
\lambda_1+\lambda_2\}$, and so on. Subgroups of $S_n$ of the form
$S_\lambda$ are called \emph{Young subgroups.} If $L$ is a subset of $S$,
then $\langle L \rangle =S_\lambda$ for a unique composition $\lambda$ of
$n$. Moreover, the rule $L\mapsto \langle L \rangle$ defines a bijection
between the set of subsets of $S$ and the set of Young subgroups of $S_n$. 
Two Young subgroups $S_\lambda$ and $S_{\lambda'}$ are conjugate if and only if
$\lambda$ and $\lambda'$ determine the same partition of $n$. In this way,
the set of shapes of $S_n$ is parametrized by partitions of $n$.

The set of conjugacy classes in $S_n$ is also parametrized by partitions of
$n$. Suppose $\lambda= (\lambda_1, \dotsm,\lambda_a)$ is a partition of
$n$. For $1\leq i\leq a$ define $c_i$ in $S_n$ by
\[
c_i(v)=
\begin{cases}
  v+1 & \text{if $v=u+1, \dots, u+\lambda_i-1$,} \\
  u+1& \text{if $v=u+\lambda_i$,} \\
  v &\text{otherwise},
\end{cases} 
\qquad\text{where}\qquad u = \sum_{k=1}^{i-1} \lambda_k.
\] 
(Here and in the formulas below, 
we use the convention that an empty sum is $0$.)
Then $c_i$ is a $\lambda_i$-cycle in the direct factor $S_{\lambda_i}$ of
$S_\lambda$. Define $w_\lambda=c_1c_2\dotsm c_a$.  Then
\begin{itemize}
\item $w_\lambda$ is a representative of the unique cuspidal conjugacy class
  in $S_{\lambda}$ and
\item $\{\, w_\lambda \mid \lambda \vdash n\,\}$ is a complete set of
  conjugacy class representatives in $S_n$.
\end{itemize}

For each $i$ such that $\lambda_i=\lambda_{i+1}$ define $x_i$ in $S_n$ by
\[
x_i(v)=
\begin{cases}
  v+\lambda_i & \text{if $v=u+1, \dots, u+\lambda_i$,} \\
  v-\lambda_i& \text{if $v=u+\lambda_i+1, \dots, u+2\lambda_i$,} \\
  v &\text{otherwise},
\end{cases} 
\quad\text{where}\quad u = \sum_{k=1}^{i-1} \lambda_k.
\]
Then conjugation by $x_i$ permutes $\{c_1, \dots, c_a \}$ by
exchanging the cycles $c_i$ and $c_{i+1}$ and hence $x_i$ centralizes
$w_\lambda$. It is well-known that $C_{S_{\lambda}}(w_\lambda)$ is the
abelian group generated by the cycles $c_1$, \dots, $c_a$ and that
$C_{S_n}(w_\lambda)$ is generated by $C_{S_{\lambda}}(w_\lambda)$
together with the involutions $x_i$, for each $i$ such that
$\lambda_i=\lambda_{i+1}$.  The abelianization of
$C_{S_n}(w_\lambda)$ is generated by the images of the first $c_i$ and
$x_i$ for every cycle length.  Whenever $i$ is a part of $\lambda$
define $\lambda(i)= \left| \{\,k \mid \lambda_k=i\,\} \right|$ and
$\ibar = \min \{\, k\mid \lambda_k = i\,\}$.
\begin{lemma}
  Let $X_{\lambda} = \{\, c_{\ibar}\mid i\in \lambda\,\} \amalg \{\,
  x_{\ibar}\mid \lambda( i) >1 \,\}$.  Suppose $\psi\colon
  X_{\lambda} \to \mathbb C^{\times}$ satisfies
  \begin{enumerate}
  \item $\psi(c_i)$ is a $\lambda_i^{\text{th}}$ root of unity for all
    $c_i \in X_{\lambda}$, and
  \item $\psi(x_i)^2=1$ for all $x_i \in X_{\lambda}$.
  \end{enumerate}
  Then $\psi$ has a unique extension to a linear character of
  $C_{S_n}(w_\lambda)$. Moreover, every linear character of
  $C_{S_n}(w_\lambda)$ arises in this way.
\end{lemma}

\begin{proof}
  See the proof of~\autoref{lem:b} below.
\end{proof}

For $k\geq 1$ denote the $k\th$ root of unity $e^{2\pi i/k}$ by $\zeta_k$.
For a partition $\lambda=(\lambda_1, \dots, \lambda_a)$ of $n$, let
$\varphi_{\lambda}^A$ be the character of $C_{S_n}(w_\lambda)$ defined (as in
the preceding lemma) by
\begin{itemize}
\item $\varphi_{\lambda}^A(c_i) =\zeta_{|c_i|}$ for all $c_i \in X_{\lambda}$ and
\item $\varphi_{\lambda}^A(x_i) =1$ for all $x_i \in X_{\lambda}$.
\end{itemize}
The next theorem is proved in~\cite{douglasspfeifferroehrle:cohomology}.

\begin{theorem}\label{thm:dpr}
  Suppose $\lambda$ is a partition of $n$ and let $\widetilde{\rho_\lambda}$
  be the top component character of the parabolic subgroup $S_\lambda$ of
  $S_n$. Then
  \[
  \widetilde{\rho_\lambda}= \operatorname{Ind}_{ C_{S_n} (w_\lambda)}^{
    N_{S_n}(S_\lambda)} \varphi_\lambda^A .
  \]
\end{theorem}

\subsection{The characters \texorpdfstring{$\varphi^B_{\mu}$}{}}

A \emph{signed partition} is a composition
\[
\mu=(\mu_1^-, \dots, \mu^-_a, \mu^+_1, \dots ,\mu^+_b),
\]
where $\mu_1^-\leq \dotsm \leq \mu^-_a$ and $\mu_1^+\geq \dotsm \geq
\mu^+_b$. Then
\[
\mu^-= (\mu_a^-, \dots, \mu^-_1)\qquad \text{and} \qquad \mu^+= (\mu^+_1,
\dots ,\mu^+_b)
\]
are partitions. We use this labeling convention for compatibility
with~\cite{geckpfeiffer:characters} and the \GAP\ functions described below,
where the conjugacy class representatives are chosen to have minimal
length. If $|\mu^-|+| \mu^+|=n$, then $\mu$ is a signed partition of $n$ and
we write $\mu \Vdash n$.

A \emph{signed permutation} of $n$ is a permutation $w\colon \pm [n]\to \pm
[n]$ such that $w(-i)=-w(i)$ for $1\leq i\leq n$. Signed permutations will
be identified with their restrictions to functions $[n]\to \pm [n]$ without
comment. Let $W_n$ denote the group of all signed permutations of $n$. Let
$S$ denote the set of transpositions $s_i= ( i\ i+1)$ for $1\leq i\leq n-1$,
together with the signed permutation $t$ defined by $t(1)=-1$ and $t(i)=i$
for $2\leq i\leq n$. Then $(W_n, S)$ is a Coxeter system of type $B_n$.

For a signed partition $\mu= (\mu_1^-, \dots, \mu^-_a, \mu^+_1, \dots
,\mu^+_b)$ of $n$ define $W_\mu$ to be the subgroup
\[
W_{\mu_1^-} \times \dotsm \times W_{\mu_a^-} \times S_{\mu_1^+} \times
\dotsm \times S_{\mu_b^+}
\]
of $W_n$. Here, a similar identification to that for the embedding of
$S_\lambda= S_{\lambda_1^+} \times \dotsm \times S_{\lambda_a}$ in $S_n$ is
used. Thus, $W_{\mu_1^-}$ acts on $\{1, \dots, \mu^-_1\}$, $W_{\mu_2^-}$
acts on $\{\mu_1^-+ 1, \dots, \mu_1^-+\mu^-_2\}$, \dots, $S_{\mu_1^+}$ acts
on $\{|\mu^-|+1, \dots, |\mu^-|+ \mu_1^+\}$, and so on.

The conjugacy classes of parabolic subgroups of $W_n$, and hence the set of
shapes of $W_n$, are parametrized by the set of partitions of $m$ for $0\leq
m\leq n$. Suppose $0\leq m\leq n$ and $\lambda$ is a partition of
$m$. Define $\mu_\lambda$ to be the signed partition $(\, (n-m), \lambda)$
of $n$ and define $W_\lambda= W_{\mu_\lambda}$. Then the subgroups
$W_\lambda$ for $\lambda$ a partition of $m$ with $0\leq m\leq n$ form a set
of representatives for the conjugacy classes of parabolic subgroups of
$W_n$.

The set of signed partitions of $n$ indexes the conjugacy classes in $W_n$
as follows. Let $\mu= (\mu_1^-, \dots, \mu^-_a, \mu^+_1, \dots ,\mu^+_b)$ be
a signed partition of $n$. For $1\leq i\leq a$, define $c_i$ in $W_n$ by
\[
c_i(v)=
\begin{cases}
  v+1 &\text{if $v=u+1, \dots, u+\mu^-_i-1$,} \\
  -(u+1)&\text{if $v=u+\mu^-_i$,} \\
  v &\text{if $v\in [n]\setminus \{u+1, \dots, u+\mu^-_i\}$,}
\end{cases} 
\quad\text{where}\quad u = \sum_{k=1}^{i-1} \mu^-_k.
\]
Then $c_i$ is a negative $\mu^-_i$-cycle in the direct factor $W_{\mu^-_i}$
of $W_\mu$. Similarly, for $1\leq j\leq b$ define $d_j$ in $W_n$ by
\begin{align*}
  d_j(v)&=
  \begin{cases}
    v+1 & \text{if $v=u+1, \dots, u+\mu^+_j-1$,} \\
    u+1& \text{if $v=u+\mu^+_j$,} \\
    v & \text{if $v\in [n]\setminus \{u+1, \dots, u+\mu^+_j\}$,}
  \end{cases} 
\text{ where } u= |\mu^-|+\sum_{k=1}^{j-1} \mu^+_k 
\end{align*}
Then $d_j$ is a positive $\mu^+_j$-cycle in the direct factor $S_{\mu^+_j}$
of $W_\mu$. Finally, define
\[
w_\mu= c_1\dotsm c_a d_1\dotsm d_b.
\]
Then $\{\, w_\mu\mid \mu \Vdash n\,\}$ is a set of conjugacy class
representatives in $W_n$.

For a signed partition $\mu$ of $n$, define $\mubar$ to be the signed
partition $(\, (|\mu^-|), \mu^+)$ of $n$. Then for a partition $\lambda$ of
$m$ with $0\leq m\leq n$, $\{\, w_{\mu}\mid \mubar= \mu_\lambda\,\}$ is a
set of representatives of the cuspidal conjugacy classes in the parabolic
subgroup $W_{\lambda}$.

Suppose $\lambda$ is a partition of $m$ with $0\leq m\leq n$ and $\mu$ is a
signed partition such that $\mubar=\mu_\lambda$. For each $i$ such that
$\mu^-_i=\mu^-_{i+1}$ define $x_i$ in $W_n$ by
\[
x_i(v)=
\begin{cases}
  v+\mu^-_i & \text{if $v=u+1, \dots, u+\mu^-_i$,} \\
  v-\mu^-_i& \text{if $v=u+\mu^-_i+1, \dots, u+2\mu^-_i$,} \\
  v &\text{if $v\in [n]\setminus \{u+1, \dots, u+2\mu^-_i\}$,}
\end{cases} 
\quad\text{where}\quad u = \sum_{k=1}^{i-1} \mu^-_k.
\]
Then conjugation by $x_i$ permutes $\{c_1, \dots, c_a ,d_1, \dots, d_b \}$
by exchanging the negative cycles $c_i$ and $c_{i+1}$ and hence $x_i$
centralizes $w_\mu$. Next, for each $j$ such that $\mu^+_j=\mu^+_{j+1}$
define $y_j$ in $W_n$ by
\begin{multline*}
  y_j(v)= \begin{cases}
    v+\mu^+_j &\text{if $v=u+1, \dots, u+\mu^+_j$,} \\
    v-\mu^+_j&\text{if $v=u+\mu^+_j+1, \dots, u+2\mu^+_j$,} \\
    v &\text{if $v\in [n]\setminus \{u+1, \dots, u+2\mu^+_j\}$,}
  \end{cases} 
\quad\text{where } u = m+\sum_{k=1}^{j-1} \mu^+_k.
\end{multline*}
Then conjugation by $y_j$ permutes $\{ c_1, \dots, c_a, d_1, \dots, d_b\}$
by exchanging the positive cycles $d_j$ and $d_{j+1}$ and hence $y_j$
centralizes $w_\mu$. Finally, for $1\leq j\leq b$ define $r_j$ in $W_n$ by
\[
r_j(v)= \begin{cases}
  -v&\text{if $v=u+1, \dots, u+\mu^+_j$,} \\
  v &\text{if $v\in [n]\setminus \{u+1, \dots, u+\mu^+_j\}$,}
  \end{cases}
  \quad\text{where} \quad u =
    m+\sum_{k=1}^{j-1} \mu^+_k.
\]
Then $r_j$ centralizes $c_i$ for $1\leq i\leq a$ and $d_k$ for $1\leq k\leq
b$ and hence centralizes $w_\mu$.

It is not hard to see that $C_{W_{\lambda}} (w_\mu)$ is generated by the
elements $c_i$, for $1\leq i\leq a$; $d_j$, for $1\leq j\leq b$; and $x_i$,
for $1\leq i\leq a$ such that $\mu^-_i=\mu^-_{i+1}$. By
\cite[Proposition~4.4]{konvalinkapfeifferroever:centralizers}, the elements
$r_j$, for $1\leq j\leq b$, and $y_j$, for $1\leq j\leq b$ with
$\mu^+_j=\mu^+_{j+1}$, generate a complement to $C_{W_\lambda} ( w_\mu)$ in
$C_{W_n} (w_\mu)$.

With the preceding notation, define the following.
Whenever $i$ is a part of $\mu^-$ define $\mu^-(i)= \left| \{\,k \mid
    \mu^-_k=i\,\} \right|$ and $\ibar =\min \{\, k\mid
  \mu^-_k=i\,\}$. Similarly, whenever $j$ is a part of $\mu^+$ define
  $\mu^+(j)= \left| \{\,k \mid \mu^+_k=j\,\} \right|$ and $\jbar=\min \{
  k\mid \mu^+_k=j \}$.  

\begin{lemma}\label{lem:b}
Let 
  \begin{multline*}
    X_\mu = \{\, c_{\ibar}\mid i\in \mu^-\,\} \amalg \{\, x_{\ibar}\mid
    \mu^-( i) >1 \,\} \amalg \{\,d_{\jbar} \mid j\in \mu^+ \,\} \amalg \{\,
    y_{\jbar }\mid \mu^+ (j) >1 \,\} \amalg \{\,r_{\jbar} \mid j\in\mu^+
    \,\}.
  \end{multline*}
  Suppose $\psi\colon X_\mu \to \BBC^{\times}$ satisfies
  \begin{enumerate}
  \item $\psi(c_i)$ is a ${2\mu^-_i}$-th root of unity for all $c_i \in X_\mu$,
  \item $\psi(x_i)^2 = 1$ for all $x_i \in X_\mu$,
  \item $\psi(d_j)$ is a ${\mu^+_j}$-th root of unity for all $d_i \in X_\mu$,
  \item $\psi(y_j)^2 = 1$ for all $y_i \in X_\mu$, and
  \item $\psi(r_i)^2 = 1$ for all $r_i \in X_\mu$.
  \end{enumerate}
  Then $\psi$ has a unique extension to a linear character of
  $C_{W_n}(w_\mu)$. Moreover, every linear character of
  $C_{W_n}(w_\mu)$ arises in this way.
\end{lemma}

\begin{proof}
  In this proof we set $C= C_{W_n}(w_\mu)$ and let $Z_{m}$ denote the
  cyclic group of order $m$. If $\lambda$ is a composition, we write $j\in
  \lambda$ when some part of $\lambda$ is equal $j$. 

  To prove the lemma, it is enough to show that the abelianization
  $C/[C,C]$ is isomorphic to
  \[ 
  \bigg(\prod_{i\in \mu^-} Z_{ 2j} \bigg) \times \bigg( \prod_{\mu^-( i)
    >1} Z_{2} \bigg) \times \bigg( \prod_{j \in \mu^+} Z_{j} \bigg) \times
  \bigg( \prod_{\mu^+( j) >1} Z_{2} \bigg) \times \bigg( \prod_{j\in \mu^+}
  Z_{2} \bigg),
  \]
  and is generated by the image of $X_\mu$.

  It is straightforward to check that when $S_m$ acts on $(Z_{u})^m$ by
  permuting the factors, the abelianization of the semidirect product
  $(Z_{u})^m \rtimes S_m$ is isomorphic to $Z_{u} \times Z_{2}$, and is
  generated by the image of a generator of the first direct factor in the
  product $(Z_{u})^m$ and the image of the transposition $s_1$. Taking
  $u=2$ we have in particular that the abelianization of $W_m$ is isomorphic
  to $Z_{2} \times Z_{2}$, and is generated by the images of $t$ and
  $s_1$.

  Similarly, when $W_m$ acts on $(Z_{u})^m$ by permuting the factors (so
  the generator $t$ and its conjugates act trivially), it is straightforward
  to check that the abelianization of the semidirect product $(Z_{u})^m
  \rtimes W_m$ is isomorphic to $Z_{u} \times Z_{2} \times Z_{2}$, and is
  generated by the image of a generator of the first direct factor in the
  product $(Z_{u})^m$, the image of the transposition $s_1$, and the image
  of $t$.

  It follows from the description of the centralizer $C$ in \cite[\S4.2]{konvalinkapfeifferroever:centralizers} that
  \begin{equation}
    \label{CommutatorQuotient}
    C\cong
    \prod_{i\in\mu^-} \left( (Z_{2i})^{\mu^-(i)} \rtimes S_{\mu^-(i)} \right)
    \times \prod_{j\in\mu^+} \left( (Z_{j})^{\mu^+(j)} \rtimes W_{\mu^+(j)}
    \right), 
  \end{equation}
  so to complete the proof it suffices to consider the abelianizations of
  the direct factors of \autoref{CommutatorQuotient} individually. It
  follows from the remarks above that for $i\in \mu^-$, the abelianization
  of $(Z_{2i})^{\mu^-(i)} \rtimes S_{\mu^-(i)}$ is isomorphic to $Z_{2i}
  \times Z_{2}$, and is generated by the images of $c_{\ibar}$ and
  $x_{\ibar}$. Similarly, for $j\in \mu^+$ the abelianization of
  $(Z_{j})^{\mu^+(j)} \rtimes W_{\mu^+(j)}$ is isomorphic to $Z_{j} \times
  Z_{2}\times Z_{2}$, and is generated by the images of $d_{\jbar}$,
  $y_{\jbar}$, and $r_{\ibar}$.
\end{proof}

For a signed partition $\mu= (\mu_1^-, \dots, \mu^-_a, \mu^+_1, \dots
,\mu^+_b)$ of $n$ let $\varphi_{\mu}^B$ be the character of $C_{W_n}(w_\mu)$
defined (as in the preceding lemma) by
\begin{itemize}
\item $\varphi_{\mu}^B(c_i) =\zeta_{2k}$ where $\mu_i^-=2^lk$ with $k$
  odd, for $1\leq i\leq a$,
\item $\varphi_{\mu}^B(d_j) =\zeta_{|d_j|}$ for $1\leq j\leq b$,
\item $\varphi_{\mu}^B(x_i) =-1$ for all $i$ such that
  $\mu^-_i =\mu^-_{i+1}$,
\item $\varphi_{\mu}^B(y_j) =1$ for all $j$ such that $\mu^+_j
  =\mu^+_{j+1}$, and
\item $\varphi_{\mu}^B(r_j) = (-1)^{\mu^+_j-1}$ for $1\leq j\leq b$.
\end{itemize}

\begin{theorem}\label{thm:b}
  Suppose $\lambda$ is a partition of $m$ with $0\leq m\leq n$ and let
  $\widetilde{\rho_\lambda}$ be the character of $N_{W_n}(W_\lambda)$
  afforded by the top component of $\BBC W_\lambda$. Then
  \[
  \widetilde{\rho_\lambda}= \sum_{\mubar=\mu_\lambda} \operatorname{Ind}_{
    C_{W_n} (w_\mu)}^{ N_{W_n}(W_\lambda)} \varphi_\mu^B
  \]
  for $n\leq 8$.
\end{theorem}

We have verified \autoref{thm:b} using the \GAP\ computer algebra system
\cite{gap3} with the \CHEVIE\ \cite{chevie} and \ZigZag\
\cite{zigzag} packages to compute both sides of the equality in the
theorem. The computation of the character $\widetilde{\rho_\lambda}$ of
$N_{W_n}(W_\lambda)$ was described in \autoref{sec:state}. The sum was
computed with the help of some \GAP\ functions. First, we defined a function 
\[
\text{\tt Lambda2Character( mu, cval, dval, xval, yval, rval )}
\] 
that takes a signed partition {\tt mu} and the character values as
in~\autoref{lem:b}, given as functions from the set of parts of $\mu$ to the
cyclotomic field, and returns the character $\psi$ of
$C_{W_n}(w_\mu)$. Second, 
we defined a function {\tt BCharacter( mu )} that takes a
signed partition, evaluates {\tt Lambda2Character} at appropriate
values, and returns the linear character $\varphi_\mu^B$ of
$C_{W_n}(w_\mu)$. 
For any given partition $\lambda$ of $m$
with
$0\leq m\leq n$, one can then use the characters  $\varphi_\mu^B$ 
with $\mubar=\mu_\lambda$ to compute  
 the sum of induced characters
$\sum_{\mubar=\mu_\lambda} \operatorname{Ind}_{ C_{W_n} (w_\mu)}^{
  N_{W_n}(W_\lambda)} \varphi_\mu^B$.

It is tempting to speculate 
that the characters $\{\, \varphi_\mu^B\mid \mu\Vdash n,\
\mubar= \mu_\lambda\,\}$ satisfy the conclusion of~\autoref{thm:1} for the
parabolic subgroup $W_{\lambda}$ of $W_n$ for all $n\ge 2$.

A similar decomposition of the regular character of the
Coxeter group of type $B_n$ 
into characters that are induced from linear characters
of element centralizers has been suggested by Bonnaf\'e
\cite[\S 10, Ques.\ (6)]{bonnafe}.
However, a straightforward calculation
shows that his decomposition is different from ours,
even for $n = 2$.

\subsection{The characters \texorpdfstring{$\varphi^D_{\mu}$}{}}

We regard the Coxeter group of type~$D_n$ for $n\geq 4$ as the reflection
subgroup $W'_n$ of the Coxeter group $W_n$ consisting of signed permutations
with an even number of sign changes, so
\[
W'_n=\{\, w\in W_n\mid \text{$|\{\, i\in [n] \mid w(i)<0\,\}| \in
  2\BBN$}\,\} .
\]
Set $t'= s_1ts_1$. Then $W'_n$ is a reflection subgroup of $W_n$ with Coxeter
generating set $S'=\{t', s_1, \dots, s_{n-1}\,\}$.

The shapes of $W'$ were determined in
\cite[Proposition~2.3.13]{geckpfeiffer:characters} as follows. First, if
$\lambda$ is a partition of $m$ with $m\leq n-2$, or if $\lambda$ is a
partition of $n$ containing at most one odd part, set $W'_\lambda=W'_n\cap
W_\lambda$. Second, if $\lambda$ is a partition of $n$ with all even parts,
then $\lambda$ indexes two conjugacy classes. One is represented by
$W'_{\lambda^+}= S_\lambda$ and the other is represented by $W'_{\lambda^-}
= t S_{\lambda}t$. 

Suppose that $\mu$ is a signed partition of $n$. If $c$ is a negative cycle,
then $\{\, i\in [n] \mid c(i)<0\,\}$ has an odd number of elements and so
$w_\mu= c_1\dotsm c_a d_1\dotsm d_b$ lies in $W'_n$ if and only if $a$ is
even, that is, if and only if $\mu^-$ has an even number of parts. Notice
that the individual negative cycles $c_i$ do not lie in $W'_n$, but the
products $c_ic_k$ do lie in $W'_n$, and hence in $C_{W'_n}(w_\mu)$.

Now suppose that $\lambda$ is a partition of $m$ with $0<m\leq n-2$, or that
$\lambda$ is a partition of $n$ with at least one odd part. Then $\{\, w_\mu
\mid \mubar= \mu_\lambda,\ l(\mu^-)\in 2\BBN\,\}$ is a complete set of
representatives for the cuspidal conjugacy classes in $W'_\lambda$. If
$\lambda$ is a partition of $n$ with all parts even, then the element
$w_\lambda$ in $S_\lambda$ represents the unique cuspidal class in
$W'_{\lambda^+}$ and the element $tw_\lambda t$ represents the unique
cuspidal class in $W'_{\lambda^-}$ (see
\cite[Proposition~3.4.12]{geckpfeiffer:characters}). Furthermore, because
$C_{W'_n}\left(w\right) =W'_n\cap C_{W_n} (w)$ for $w$ in $W'_n$, we can
define linear characters of $C_{W'_n}(w)$ simply by restricting characters
of $C_{W_n}(w)$.

The conclusion of~\autoref{thm:1} has been shown to hold whenever $W_L$ is a
product of symmetric groups
in~\cite{douglasspfeifferroehrle:cohomology}. Thus, to simplify the
exposition, in the following we consider only the parabolic subgroups
$W_\lambda$ where $\lambda$ is a partition of $m$ with $0\leq m\leq n-2$.

Suppose $\mu$ is a signed partition of $n$ such that $l(\mu^-)$ is even. It
follows from~\autoref{lem:b} that there is a linear character
$\psi_\mu$ of $C_{W_n}(w_\mu)$ such that
\begin{itemize}
\item $\psi_\mu(c_i) =\zeta_{|c_i|}$ for $1\leq i\leq a$,
\item $\psi_\mu(d_j) =\zeta_{|d_j|}$ for $1\leq j\leq b$,
\item $\psi_\mu(x_i) =-1$ for all $i$ such that $\mu^-_i=\mu^-_{i+1}$,
\item $\psi_\mu(y_j) =1$ for all $j$ such that $\mu^+_j=\mu^+_{j+1}$, and
\item $\psi_\mu(r_j) = -1$ for $1\leq j\leq a$.
\end{itemize}
Define $\varphi^D_\mu$ to be the restriction of $\psi_\mu$ to
$C_{W'_n}(w_\mu)$. As we have already observed, the individual negative
cycles $c_i$ do not lie in $W_n'$. Similarly, if $\mu^+_j$ is odd, then
$r_j$ is not in $W_n'$. We have
\begin{itemize}
\item $\varphi^D_\mu (c_ic_k) =\zeta_{|c_i|} \zeta_{|c_k|}$ for $1\leq
  i,k\leq a$,
\item $\varphi^D_\mu (d_j) =\zeta_{|d_j|}$ for $1\leq j\leq b$,
\item $\varphi^D_\mu(x_i) =-1$ for all $i$ such that $\mu^-_i=\mu^-_{i+1}$,
\item $\varphi^D_\mu(y_j) =1$ for all $j$ such that $\mu^+_j=\mu^+_{j+1}$, and
\item $\varphi^D_\mu (r_j) = -1$ for $1\leq j\leq b$ such that $\mu^+_j$ is
  even, and
\item $\varphi^D_\mu (r_i r_k) = 1$ for $1\leq i,k\leq b$ such that $\mu^+_i$ and $\mu^+_k$ are odd.
\end{itemize}

\begin{theorem}\label{thm:d}
  Suppose $\lambda$ is a partition of $m$ with $0\leq m\leq n-2$ and let
  $\widetilde{\rho_\lambda}$ be the character of $N_{W_n'}(W_\lambda')$
  afforded by the top component of $\BBC W_\lambda'$. Then
  \[
  \widetilde{\rho_\lambda}= \sum_{\substack{\mubar=\mu_\lambda\\ l(\mu^-)\in
      2\BBN}} \operatorname{Ind}_{ C_{W_n'} (w_\mu)}^{ N_{W_n'}(W_\lambda')}
  \varphi_\mu^D
  \]
  for $n\leq 8$.
\end{theorem}

The verification of ~\autoref{thm:d} parallels the verification
of~\autoref{thm:b}. The computation of the character
$\widetilde{\rho_\lambda}$ of $N_{W_n}(W_\lambda)$ was described in
\autoref{sec:state}. The sum was computed with the help of further 
\GAP\  functions. We defined
a function {\tt DCharacter( mu )} that takes a signed partition 
{\tt mu}, evaluates the above function 
{\tt Lambda2Character} at appropriate values, and returns the
linear character 
$\psi_\mu$ of $C_{W_n}(w_\mu)$.
Restriction  to
$C_{W_n'}(w_\mu)$ yields  the linear character $\varphi_\mu^D$.
For any given partition $\lambda$ of $m$
with
$0\leq m\leq n-2$, one can then use the characters  $\varphi_\mu^D$ 
with $\mubar=\mu_\lambda$ to compute  
the sum of induced characters $\sum_{\substack{\mubar=\mu_\lambda\\
    l(\mu^-)\in 2\BBN}} \operatorname{Ind}_{ C_{W_n'} (w_\mu)}^{
  N_{W_n'}(W_\lambda')} \varphi_\mu^D$.

It is tempting to speculate 
that the characters $\varphi_\mu^D$ satisfy the conclusion
of~\autoref{thm:1} for $W'_n$ for all $n$ and all partitions $\lambda$ of
$m$ with $0\leq m\leq n-2$.
Note that, in general, the characters $\varphi_\mu^D$
are different from the restrictions of the characters $\varphi_\mu^B$
to the centralizers in $W_n'$.

\section{Exceptional groups}\label{sec:exc}

As in~\cite[\S4]{bishopdouglasspfeifferroehrle:computationsII}, we only need
to consider subsets $L$ of $S$ (up to conjugacy) such that $W_L$ is not
bulky in $W$, $W_L$ has rank at least three, and $W_L$ is not a direct
product of Coxeter groups of type~$A$. The pairs $(W, W_L)$ are given by
type in the following table.
\begin{equation*}
  \begin{array}{c|l}
    \hline W&W_L\\\hline
    E_7&D_4,\,A_1D_4,\,D_5,\,A_1D_5,\,E_6\\
    E_8&D_4,\,A_1D_4,\,D_5,\,A_2D_4,\,A_1D_5,\,D_6,\,E_6,\,A_2D_5,
    \,A_1D_6,D_7\\\hline 
  \end{array}
\end{equation*}

Because of space considerations, we do not list the values of the characters
$\widetilde{\rho_L}= \alpha_L\epsilon \widetilde{\omega_L}$. Instead we list
the characters $\varphi_w$ that satisfy the conclusion of \autoref{thm:1}
for each pair $(W,L)$ when $L$ is a proper subset of $S$ in
\autoref{E7CTable} and \autoref{E8CTable}.  However, see
\cite[\S3.1]{bishopdouglasspfeifferroehrle:computationsII} for an example
with all the data included.

In the tables we exhibit a set of generators of the centralizer of each
cuspidal class representative of $W_L$.  We use the symbol $w_i$ to denote a
representative of the $i^{\text{th}}$ conjugacy class of a group in the list
of conjugacy classes returned by the command {\ttfamily ConjugacyClasses} in
\GAP\ \cite{gap3}.  At each generator of $C_W\left(w_i\right)$ we display the
value of the character $\varphi_{w_i}$, denoted simply by $\varphi_i$. The
symbol $w_0$ represents the longest element of $W$, while $w_L$ represents
the longest element of $W_L$ when $L$ is a proper subset of $S$.  We use the
symbols $1,2,\ldots,n$ to denote the elements of $S$.  For $p\ge 1$ we
denote the $p^\text{th}$ root of unity $e^{2\pi i/p}$ by $\zeta_p$. Finally,
$r$ represents the reflection with respect to the highest long root in the
root system of $W$. We sometimes express generators in terms of longest
elements of certain parabolic subgroups of $W$.  For this purpose, we fix
the following subsets of $S$.
\begin{alignat*}{2}
  E&=\left\{2,3,4,5\right\},\qquad
  &F&=\left\{1,2,3,4,5\right\},\\
  G&=\left\{2,3,4,5,6\right\},
  &H&=\left\{1,2,3,4,5,6\right\},\\
  I&=\left\{2,3,4,5,6,7\right\}, &J&=\left\{2,3,4,5,6,7,8\right\}.
\end{alignat*}

For the subgroup of type~$E_6$ of $W\left(E_8\right)$ we modified the
cuspidal conjugacy class representatives to match those used in
\cite{bishopdouglasspfeifferroehrle:computationsII}.  Namely, we took
$w_{15}=123456$, $w_{14}=24w_{15}$, $w_{12}=13456r_{E_6}$,
$w_{10}=w_{15}^2$, and $w_4=12356r_{E_6}$ where $r_{E_6}$ is the reflection
corresponding with the highest root in the $E_6$ subsystem.  For the
subgroup of type $A_1E_6$ of $W\left(E_8\right)$ the representatives
$\left(w_8,w_{20},w_{24}, w_{28},w_{30}\right)$ were obtained from the
representatives $\left(w_4, w_{10}, w_{12}, w_{14}, w_{15}\right)$ for the
subgroup of type $E_6$ by multiplying by the generator~$8$.

\subsection{Proof of \autoref{thm:1} when \texorpdfstring{$W$}{W} has type
  \texorpdfstring{$E$}{E} and \texorpdfstring{$L=S$}{L=S}}
\label{sec:e2}

To prove~\autoref{thm:1} when $W$ has type $E$ and $L=S$ we proceed as in
the case when $L$ is a proper subset of $S$ except that we use the methods
described in \autoref{sec:state} to compute the Orlik-Solomon character
$\widetilde{\omega_S}=\omega_S$.  Again, we do not list the values of the
characters $\widetilde{\rho_L}= \alpha_L\epsilon \widetilde{\omega_L}$. In
\autoref{E7Table} and \autoref{E8Table} we list the characters $\varphi_d$
that satisfy the conclusion of~\autoref{thm:1} when $L=S$ for the groups of
types~$E_7$ and $E_8$ respectively. Here the conjugacy classes are labeled
by Carter diagrams \cite{carter:conjugacy} and we denote the character
$\varphi_{w_d}$ by $\varphi_d$ where $d$ is a Carter diagram.

As in \cite{bishopdouglasspfeifferroehrle:computations,
  bishopdouglasspfeifferroehrle:computationsII}, we give additional
information about regular conjugacy classes.  If $w$ is in $W$ and $\zeta$
is an eigenvalue of $w$ on $V$, then we denote the determinant of the
representation of $C_W\left(w\right)$ on the $\zeta$-eigenspace of $w$ by
$\det |_\zeta$. By Springer's theory of regular elements
\cite{springer:regular} the centralizer $C_W\left(w\right)$ is a complex
reflection group acting on an eigenspace of $w$ whenever $w$ is a regular
element. In each table we indicate which classes are regular in the column
labeled $\mathrm{Reg}$, which is to be interpreted as follows. If $w$ is
regular and $\varphi_w$ is a power of $\det |_\zeta$ for some $\zeta$ then
we indicate this power in the Reg column.  However, if $w$ is regular but
$\varphi_w$ is {\em not} a power of $\det |_\zeta$ for any $\zeta$, then we
indicate this by writing $\spadesuit$ in the Reg column.  Whenever practical
we describe the structure of $C_W\left(w\right)$ in terms of $Z_m$ and
$S_m$, which denote the cyclic group of size $m$ and the symmetric group on
$m$ letters.

When $C_W\left(w_d\right)$ acts as a complex reflection group on an
eigenspace of $w_d$ we often specify the character $\varphi_d$ by listing
its values on generators of its ``Dynkin diagram'' presentation
\cite{brouemallerouquier:complex}. We also exhibit its Dynkin diagram in
these cases. Note that vertices of the Dynkin diagram connected by single
edges are conjugate in $W$ and thereby take the same character values, which
we list only once.

For the group $W\left(E_8\right)$ we have slightly modified the
representatives of the cuspidal conjugacy classes supplied by
\textsf{GAP}. Namely, we observe that $w_{E_8\left(a_8\right)}$ can be taken
to be $w_0w_{A_2^4}$ and then $C_W\left( w_{E_8\left(a_8\right)} \right)=C_W
\left(w_{A_2^{\smash4}} \right)$.  Similar observations hold for the classes
$A_4^2$ and $E_8\left(a_6\right)$.

\appendix

\section{Tables}\label{TableSection}

\begin{longtable}{ccl}
  \caption{$W=W\left(E_7\right)$, $L\subsetneq S$}\\
  $L$&Type&Characters\\\endfirsthead
  \caption[]{(continued)}\\
  $L$&Type&Characters\\\endhead\toprule $\left\{2,3,4,5\right\}$&$D_4$
  \label{E7CTable}
  &$\varphi_3:\left(2,4,7,w_F,w_G\right)
  \mapsto\left(-1,-1,1,1,1\right)$\\*
  &&$\varphi_9:\left(7,2w_G,245w_H\right)
  \mapsto\left(1,\zeta_4,\zeta_4\right)$\\*
  &&$\varphi_{11}:\left(w_{11},7,2w_F2,w_G\right)
  \mapsto\left(\zeta_3,1,1,1\right)$\\\hline
  $\left\{2,3,4,5,7\right\}$&$A_1D_4$
  &$\varphi_6:\left(3,4,5,7,r,w_F\right)
  \mapsto\left(-1,-1,-1,-1,1,1\right)$\\*
  &&$\varphi_{18}:\left(w_{18},23,r,w_F234\right)
  \mapsto\left(1,-1,1,\zeta_4\right)$\\*
  &&$\varphi_{22}:\left(2345,7,r,2w_F2\right)
  \mapsto\left(\zeta_3,-1,1,1\right)$\\\hline
  $\left\{2,3,4,5,6\right\}$&$D_5$
  &$\varphi_7:\left(w_7,2,3,4,r,w_0\right)\mapsto
  \left(\zeta_4,-1,-1,-1,1,-1\right)$\\*
  &&$\varphi_{15}:\left(w_{15},r,w_0\right)\mapsto
  \left(\zeta_{12},1,-1\right)$\\*
  &&$\varphi_{17}:\left(w_{17},r,w_0\right)\mapsto
  \left(\zeta_8,1,-1\right)$\\\hline
  $\left\{1,2,3,4,5,7\right\}$&$A_1D_5$
  &$\varphi_{14}:\left(w_{14},2,4,5,7,w_0\right)\mapsto
  \left(\zeta_4,-1,-1,-1,-1,1\right)$\\*
  &&$\varphi_{30}:\left(w_{30},7,w_0\right)\mapsto
  \left(\zeta_{12},-1,1\right)$\\*
  &&$\varphi_{34}:\left(w_{34},7,w_0\right)\mapsto
  \left(\zeta_8,-1,1\right)$\\\hline
  $\left\{1,2,3,4,5,6\right\}$&$E_6$
  &$\varphi_4:\left(2343,134256,w_0\right)\mapsto
  \left(\zeta_3,1,1\right)$\\*
  &&$\varphi_{10}:\left(w_{15}^2,123654,w_0\right)
  \mapsto\left(\zeta_3,1,1\right)$\\*
  &&$\varphi_{12}:\left(w_{12},2,4,w_0\right)\mapsto
  \left(\zeta_3,-1,-1,1\right)$\\*
  &&$\varphi_{14}:\left(w_{14},w_0\right)\mapsto
  \left(\zeta_9,1\right)$\\*
  &&$\varphi_{15}:\left(w_{15},w_0\right)\mapsto
  \left(-1,1\right)$\\\toprule
\end{longtable}

\begin{longtable}{ccl}
  \caption{$W=W\left(E_8\right)$, $L\subsetneq S$}\\
  $L$&Type&Characters\\\endfirsthead
  \caption[]{(continued)}\\
  $L$&Type&Characters\\\endhead\toprule
  \label{E8CTable}
  $\left\{2,3,4,5\right\}$&$D_4$
  &$\varphi_{3}:\left(2,4,7,8,w_F,w_G\right)
  \mapsto\left(-1,-1,1,1,1,1\right)$\\*
  &&$\varphi_{9}:\left(7,8,2w_G,254w_F\right)
  \mapsto\left(1,1,\zeta_4,-\zeta_4\right)$\\*
  &&$\varphi_{11}:\left(2345,7,8,25w_F,w_G\right)
  \mapsto\left(\zeta_3,1,1,1,1\right)$\\\hline
  $\left\{2,3,4,5,8\right\}$&$A_1D_4$
  &$\varphi_{6}:\left(3,4,7,w_F,w_J,w_I\right)
  \mapsto\left(-1,-1,-1,1,-1,-1\right)$\\*
  &&$\varphi_{18}:\left(r,254w_F,3w_J\right)
  \mapsto\left(1,\zeta_4,-\zeta_4\right)$\\*
  &&$\varphi_{22}:\left(2345,r,25w_F,w_J\right)
  \mapsto\left(\zeta_3,1,1,-1\right)$\\\hline
  $\left\{1,2,3,4,5\right\}$&$D_5$
  &$\varphi_{7}:\left(2,4,1342543,7,8,w_0,r\right)
  \mapsto\left(-1,-1,\zeta_4,1,1,-1,1\right)$\\*
  &&$\varphi_{15}:\left(1234254,7,8,r,w_0\right)
  \mapsto\left(\zeta_{12},1,1,1,-1\right)$\\*
  &&$\varphi_{17}:\left(13425,7,8,r,w_0\right)
  \mapsto\left(\zeta_8,1,1,1,-1\right)$\\\hline
  $\left\{2,3,4,5,7,8\right\}$&$A_2D_4$
  &$\varphi_{9}:\left(3,4,78,w_F,w_J\right)
  \mapsto\left(-1,-1,\zeta_3,1,1\right)$\\*
  &&$\varphi_{27}:\left(78,254w_F,3w_J\right)
  \mapsto\left(\zeta_3,\zeta_4,\zeta_4\right)$\\*
  &&$\varphi_{33}:\left(2345,78,25w_F,w_J\right)
  \mapsto\left(\zeta_3,\zeta_3,1,1\right)$\\\hline
  $\left\{1,2,3,4,5,7\right\}$&$A_1D_5$
  &$\varphi_{14}:\left(2,4,1342543,7,w_0,r\right)
  \mapsto\left(-1,-1,\zeta_4,-1,1,1\right)$\\*
  &&$\varphi_{30}:\left(1234254,7,w_0,r\right)
  \mapsto\left(\zeta_{12},-1,1,1\right)$\\*
  &&$\varphi_{34}:\left(13425,7,r,w_0\right)
  \mapsto\left(\zeta_8,-1,1,1\right)$\\\hline
  $\left\{2,3,4,5,6,7\right\}$&$D_6$
  &$\varphi_{4}:\left(2,4,5,6,7,r,w_J\right)
  \mapsto\left(-1,-1,-1,-1,-1,1,1\right)$\\*
  &&$\varphi_{14}:\left(2,54234,6576,r,w_J\right)
  \mapsto\left(-1,\zeta_4,-1,1,1\right)$\\*
  &&$\varphi_{21}:\left(2,4,76542345,r,w_J\right)
  \mapsto\left(-1,-1,\zeta_3,1,1\right)$\\*
  &&$\varphi_{27}:\left(543654765,43w_J,r\right)
  \mapsto\left(-1,\zeta_3,1\right)$\\*
  &&$\varphi_{33}:\left(24234567,3w_J,r\right)
  \mapsto\left(\zeta_8,\zeta_4,1\right)$\\*
  &&$\varphi_{35}:\left(234567,r,w_J\right)
  \mapsto\left(\zeta_5,1,1\right)$\\\hline
  $\left\{1,2,3,4,5,6\right\}$&$E_6$
  &$\varphi_{4}:\left(56,14234542,8,r,w_0\right)
  \mapsto\left(\zeta_3,\zeta_3^2,1,1,1\right)$\\*
  &&$\varphi_{10}:\left(w_{15},234543,8,r,w_0\right)
  \mapsto\left(1,\zeta_3,1,1,1\right)$\\*
  &&$\varphi_{12}:\left(w_{12},2345432,r_{E_6},8,r,w_0\right)
  \mapsto\left(\zeta_3,-1,-1,1,1,1\right)$\\*
  &&$\varphi_{14}:\left(w_{14},8,r,w_0\right)
  \mapsto\left(\zeta_9,1,1,1\right)$\\*
  &&$\varphi_{15}:\left(w_{15},8,r,w_0\right)
  \mapsto\left(-1,1,1,1\right)$\\\hline
  $\left\{1,2,3,4,5,7,8\right\}$&$A_2D_5$
  &$\varphi_{21}:\left(2,4,1342543,78,w_0\right)
  \mapsto\left(-1,-1,\zeta_4,\zeta_3,-1\right)$\\*
  &&$\varphi_{45}:\left(1234254,78,w_0\right)
  \mapsto\left(\zeta_{12},\zeta_3,-1\right)$\\*
  &&$\varphi_{51}:\left(13425,78,w_0\right)
  \mapsto\left(\zeta_8,\zeta_3,-1\right)$\\\hline
  $\left\{1,2,3,4,5,6,8\right\}$&$A_1E_6$
  &$\varphi_{8}:\left(56,14234542,8,w_0\right)
  \mapsto\left(\zeta_3,\zeta_3^2,-1,-1\right)$\\*
  &&$\varphi_{20}:\left(w_{30},234543,w_0\right)
  \mapsto\left(-1,\zeta_3,-1\right)$\\*
  &&$\varphi_{24}:\left(w_{24},2345432,r_{E_6},8,w_0\right)
  \mapsto\left(\zeta_6,-1,-1,-1,-1\right)$\\*
  &&$\varphi_{28}:\left(24w_{30},w_0\right)
  \mapsto\left(\zeta_{18},-1\right)$\\*
  &&$\varphi_{30}:\left(w_{30},8,w_0\right)
  \mapsto\left(1,-1,-1\right)$\\\hline
  $\left\{2,3,4,5,6,7,8\right\}$&$D_7$
  &$\varphi_{10}:\left(w_{10},2,3,4,5,6,w_0\right)
  \mapsto\left(\zeta_4,-1,-1,-1,-1,-1,-1\right)$\\*
  &&$\varphi_{20}:\left(234,5465,7687,w_0\right)
  \mapsto\left(\zeta_4,-1,-1,-1\right)$\\*
  &&$\varphi_{31}:\left(w_{31},2,42345,w_0\right)
  \mapsto\left(\zeta_{12},-1,-\zeta_4,-1\right)$\\*
  &&$\varphi_{41}:\left(w_{41},2,3,4,w_0\right)
  \mapsto\left(\zeta_8,-1,-1,-1,-1\right)$\\*
  &&$\varphi_{47}:\left(w_{47},w_0\right)
  \mapsto\left(\zeta_{24},-1\right)$\\*
  &&$\varphi_{52}:\left(w_{52},w_0\right)
  \mapsto\left(\zeta_{20},-1\right)$\\*
  &&$\varphi_{54}:\left(w_{54},w_0\right)
  \mapsto\left(\zeta_{12},-1\right)$\\*\toprule
\end{longtable}

\begin{longtable}[H]{ccccc}
  \caption{$W=W\left(E_7\right)$, $L=S$}\\
  $d$&$C_W\left(w_d\right)$ &Gen&$\varphi_d$&Reg\\\endfirsthead
  \caption[]{(continued)}\\
  $d$&$C_W\left(w_d\right)$ &Gen&$\varphi_d$&Reg\\\endhead\toprule
  \label{E7Table}
  $A_1^7$&$W$&$S$&$\epsilon$&$\det|_{-1}$\\\hline
  $A_1^3D_4$&&$w_{A_1^3D_4}$&$\zeta_6$\\*
  &&$2,4,5,6,7$&$-1$\\\hline $E_7(a_4)$&$G_{26}$&
  $\vcenter{\begin{xy}<.5cm,0cm>:
      (0,0)+<4pt,0pt>;(2,0)-<4pt,0pt>**\dir2{-};
      (2,0)+<4pt,0pt>;(4,0)-<4pt,0pt>**\dir{-};
      (0,0)*{_2}*\xycircle<4pt>{-}; (2,0)*{_3}*\xycircle<4pt>{-};
      (4,0)*{_3}*\xycircle<4pt>{-};
    \end{xy}}$
  &$-1,\zeta_3,\zeta_3$&$\spadesuit$\\\hline
  $A_1D_6(a_2)$&&$w_{A_1D_6(a_2)}$&$-\zeta_3$\\*
  &&$2$&$-1$\\*
  &&$1342654231456$&$-1$\\*
  &&$\left(76\right)^{542345}$&$\zeta_3$\\\hline
  $A_1A_3^2$&&$w_{E_7\left(a_2\right)}$&$1$\\*
  &&$42543$&$\zeta_4$\\\hline
  $A_1D_6$&&$w_{A_1D_6}$&$\zeta_{10}$\\*
  &&$2$&$-1$\\*
  &&$4$&$-1$\\\hline
  $A_2A_5$&&$w_{A_2A_5}$&$-1$\\*
  &&$34$&$\zeta_3$\\*
  &&$2345$&$\zeta_3^2$\\\hline $E_7(a_1)$&$Z_{14}$
  &$\vcenter{\begin{xy}<.5cm,0cm>: (0,0)*{_{14}}*\xycircle<5pt>{-};
    \end{xy}}$
  &$\zeta_{14}$&$\det\vert_{\zeta_{14}}$\\\hline
  $A_7$&$Z_8\times Z_2\times Z_2$&$w_{A_7}$&$\zeta_8$\\*
  &&$w_0$&$-1$\\*
  &&$2$&$-1$\\\hline $E_7$&$Z_{18}$ &$\vcenter{\begin{xy}<.5cm,0cm>:
      (0,0)*{_{18}}*\xycircle<5pt>{-};
    \end{xy}}$
  &$\zeta_{18}$&$\det\vert_{\zeta_{18}}$\\\hline
  $E_7(a_2)$&$Z_{12}\times Z_2$&$w_{E_7\left(a_2\right)}$&$1$\\*
  &&$w_0$&$-1$\\\hline
  $E_7(a_3)$&$Z_{30}$&$w_{E_7\left(a_3\right)}$&$\zeta_{30}$\\\toprule
\end{longtable}

\begin{longtable}[H]{ccccc}
  \caption{$W=W\left(E_8\right)$, $L=S$}\\
  $d$&$C_W\left(w_d\right)$ &Gen&$\varphi_d$&Reg\\\endfirsthead
  \caption[]{(continued)}\\
  $d$&$C_W\left(w_d\right)$ &Gen&$\varphi_d$&Reg\\\endhead\toprule
  \label{E8Table}
  $A_1^8$&$W$&$S$&$\epsilon$&$\det|_{-1}$\\\hline $D_4(a_1)^2$&$G_{31}$
  &$\vcenter{\begin{xy}<.5cm,0cm>:
      (0,0)+<4pt,0pt>;(2,0)-<4pt,0pt>**\dir{-};
      (0,0)+<3pt,3pt>;(1,1)-<3pt,3pt>**\dir{-};
      (2,0)+<4pt,0pt>;(4,0)-<4pt,0pt>**\dir{-};
      (4,0)+<-3pt,3pt>;(3,1)+<3pt,-3pt>**\dir{-};
      (0,0)*{_2}*\xycircle<4pt>{-}; (2,0)*{_2}*\xycircle<4pt>{-};
      (1,1)*{_2}*\xycircle<4pt>{-}; (4,0)*{_2}*\xycircle<4pt>{-};
      (3,1)*{_2}*\xycircle<4pt>{-}; (2,1)*\xycircle<10pt>{-};
    \end{xy}}$ &$-1$&$\spadesuit$\\\hline
  $A_1^4D_4$&&$1$&$-1$\\*
  &&$w_{A_1^4D_4}$&$\zeta_3$\\*
  &&$w_{A_1E_7(a_4)}$&$\zeta_3^2$\\*
  &&$w_{A_3D_5(a_1)}$&$\zeta_3$\\\hline $A_2^4$&$G_{32}$&
  $\vcenter{\begin{xy}<.5cm,0cm>:
      (0,0)+<4pt,0pt>;(2,0)-<4pt,0pt>**\dir{-};
      (2,0)+<4pt,0pt>;(4,0)-<4pt,0pt>**\dir{-};
      (4,0)+<4pt,0pt>;(6,0)-<4pt,0pt>**\dir{-};
      (0,0)*{_3}*\xycircle<4pt>{-}; (2,0)*{_3}*\xycircle<4pt>{-};
      (4,0)*{_3}*\xycircle<4pt>{-}; (6,0)*{_3}*\xycircle<4pt>{-}
    \end{xy}}$&$\zeta_3$&$\spadesuit$\\\hline $E_8(a_8)$&$G_{32}$&
  $\vcenter{\begin{xy}<.5cm,0cm>:
      (0,0)+<4pt,0pt>;(2,0)-<4pt,0pt>**\dir{-};
      (2,0)+<4pt,0pt>;(4,0)-<4pt,0pt>**\dir{-};
      (4,0)+<4pt,0pt>;(6,0)-<4pt,0pt>**\dir{-};
      (0,0)*{_3}*\xycircle<4pt>{-}; (2,0)*{_3}*\xycircle<4pt>{-};
      (4,0)*{_3}*\xycircle<4pt>{-}; (6,0)*{_3}*\xycircle<4pt>{-}
    \end{xy}}$&$\zeta_3$&$\spadesuit$\\\hline
  $A_1E_7(a_4)$&&$2$&$-1$\\*
  &&$w_{A_1^4D_4}$&$\zeta_3$\\*
  &&\footnote[2]{$y=13427654234567876$}
  $y$&$-\zeta_3$\\\hline
  $D_4^2$&&$2$&$-1$\\*
  &&$w_E$&$1$\\*
  &&\footnote[3]{$z=1342543654276548765$} 
  $z$&$\zeta_6$\\\hline
  $A_1^2A_3^2$&&$w_{A_2E_6(a_2)}$&$1$\\*
  &&$5$&$-1$\\*
  &&$7w_G$&$\zeta_4$\\\hline $D_8(a_3)$&$G_9$&
  $\vcenter{\begin{xy}<.5cm,0cm>:
      (0,0)+<4pt,0pt>;(2,0)-<4pt,0pt>**\dir3{-};
      (0,0)*{_4}*\xycircle<4pt>{-}; (2,0)*{_2}*\xycircle<4pt>{-}
    \end{xy}}$&$\zeta_4,-1$&$\det|_{\zeta_8}$\\\hline
  $A_1^2D_6$&$Z_{10}\times S_5$&$w_{A_1^2D_6}$&$\zeta_5$\\*
  &&$3,4,5,6$&$-1$\\\hline $A_4^2$&$G_{16}$&
  $\vcenter{\begin{xy}<.5cm,0cm>:
      (0,0)+<4pt,0pt>;(2,0)-<4pt,0pt>**\dir{-};
      (0,0)*{_5}*\xycircle<4pt>{-}; (2,0)*{_5}*\xycircle<4pt>{-}
    \end{xy}}$&$\zeta_5$&$\det|_{\zeta_5}$\\\hline
  $E_8(a_6)$&$G_{16}$& $\vcenter{\begin{xy}<.5cm,0cm>:
      (0,0)+<4pt,0pt>;(2,0)-<4pt,0pt>**\dir{-};
      (0,0)*{_5}*\xycircle<4pt>{-}; (2,0)*{_5}*\xycircle<4pt>{-}
    \end{xy}}$&$\zeta_5$&$\det|_{\zeta_{10}}$\\\hline
  $A_2E_6(a_2)$&&$w_{A_2E_6(a_2)}$&$\zeta_3$\\*
  &&$w_{A_1^2A_3^2}$&$1$\\*
  &&$24$&$\zeta_3$\\*
  &&$2345$&$\zeta_3^2$\\*
  &&$87w_I$&$\zeta_3$\\\hline $E_8(a_3)$&$G_{10}$&
  $\vcenter{\begin{xy}<.5cm,0cm>:
      (0,0)+<4pt,0pt>;(2,0)-<4pt,0pt>**\dir{=};
      (0,0)*{_4}*\xycircle<4pt>{-}; (2,0)*{_3}*\xycircle<4pt>{-}
    \end{xy}}$&$-1,\zeta_3$&$\left(\det|_{\zeta_{12}}\right)^2$\\\hline
  $A_1A_2A_5$&&$w_{A_1A_2A_5}$&$1$\\*
  &&$7$&$-1$\\*
  &&$8$&$-1$\\*
  &&$34$&$\zeta_3$\\*
  &&$2345$&$\zeta_3^2$\\\hline
  $D_8(a_1)$&$Z_{12}\times S_3$&$w_{D_8(a_1)}$&$\zeta_3$\\*
  &&$48$&$-1$\\*
  &&$4578$&$1$\\\hline
  $D_8$&$Z_{14}\times Z_2$&$w_{D_8}$&$\zeta_7$\\*
  &&$2$&$-1$\\\hline
  $A_1A_7$&&$w_{A_1A_7}$&$\zeta_8$\\*
  &&$2$&$-1$\\*
  &&$34254$&$\zeta_4$\\\hline
  $A_1E_7$&&$w_{A_1E_7}$&$\zeta_9$\\*
  &&$3$&$-1$\\*
  &&$4$&$-1$\\\hline
  $A_8$&$Z_{18}\times Z_3$&$w_{A_8}$&$\zeta_9$\\*
  &&$13w_0$&$\zeta_3$\\\hline
  $E_8(a_4)$&$Z_{18}\times Z_3$&$w_{E_8(a_4)}$&$\zeta_9$\\*
  &&$34$&$\zeta_3$\\\hline $E_8(a_2)$&$Z_{20}$
  &$\vcenter{\begin{xy}<.5cm,0cm>: (0,0)*{_{20}}*\xycircle<5pt>{-};
    \end{xy}}$
  &$\zeta_5$&$\left(\det|_{\zeta_{20}}\right)^4$\\\hline
  $A_3D_5(a_1)$&&$w_{A_3D_5(a_1)}$&$\zeta_6$\\*
  &&$354$&$\zeta_4$\\*
  &&$136542$&$-1$\\\hline
  $A_2E_6$&&$w_{A_2E_6}$&$\zeta_6$\\*
  &&$34$&$\zeta_3$\\*
  &&$2435$&$\zeta_3$\\\hline
  $E_8(a_7)$&&$w_{E_8(a_7)}$&$\zeta_6$\\*
  &&$2354$&$\zeta_3$\\*
  &&$2454$&$\zeta_3$\\\hline
  $A_1E_7(a_2)$&&$w_{A_1E_7(a_2)}$&$-1$\\*
  &&$2,4,w_0$&$-1$\\\hline $E_8(a_1)$&$Z_{24}$
  &$\vcenter{\begin{xy}<.5cm,0cm>: (0,0)*{_{24}}*\xycircle<5pt>{-};
    \end{xy}}$ &$\zeta_{12}$&$\left(\det|_{\zeta_{24}}\right)^2$\\\hline
  $D_8(a_2)$&$Z_{30}\times Z_2$&$w_{D_8(a_2)}$&$\zeta_{15}$\\*
  &&$5$&$-1$\\\hline $E_8(a_5)$&$Z_{30}$
  &$\vcenter{\begin{xy}<.5cm,0cm>: (0,0)*{_{30}}*\xycircle<5pt>{-};
    \end{xy}}$ &$\zeta_{15}$&$\left(\det|_{\zeta_{15}}\right)^2$\\\hline
  $E_8$&$Z_{30}$
  &$\vcenter{\begin{xy}<.5cm,0cm>: (0,0)*{_{30}}*\xycircle<5pt>{-};
    \end{xy}}$ &$\zeta_{15}$&$\left(\det|_{\zeta_{30}}\right)^2$\\\toprule
\end{longtable}

\bigskip 

{\bf Acknowledgments}: The authors would like to acknowledge support from
the DFG-priority program SPP1489 {\em Algorithmic and Experimental Methods
  in Algebra, Geometry, and Number Theory.} This work was partially
supported by a grant from the Simons Foundation (Grant \#245399 to
J. Matthew Douglass).

\bigskip


\bibliographystyle{plainnat}

\end{document}